\documentclass[letterpaper, 10 pt, conference]{ieeeconf}  

\IEEEoverridecommandlockouts                              

\usepackage[dvips]{graphicx}
\usepackage{makeidx}
\usepackage{amsmath}
\usepackage{amsfonts}
\usepackage{amssymb}

\newcommand{\Real}{\mathbb{R}}

\newcommand{\norm}[1]{\left\Vert#1\right\Vert}

\newcommand {\ess}{{\mathrm{ess}}}

\newcommand{\pd}{{\partial}}
\newcommand{\bfA}{A}
\newcommand{\bfB}{b}
\newcommand{\bfC}{C}
\newcommand{\bfP}{P}
\newcommand{\bfQ}{Q}

\newcommand{\phivec}{\phi}

\newtheorem{lem}{Lemma}
\newtheorem{thm}{Theorem}
\newtheorem{assume}{Assumption}
\newtheorem{cor}{Corollary}

\newtheorem{example}{\it Example}

\title{\LARGE \bf  Further Results on  Lyapunov-Like Conditions of Forward Invariance and Boundedness for a Class of Unstable Systems}

\author{Alexander N. Gorban$^{1}$, Ivan Yu. Tyukin$^{2}$, Henk Nijmeijer$^{3}$ 
\thanks{$^{1}$Alexander N. Gorban is with the University of Leicester, Department of Mathematics, Leicester, University Road, LE1 7RH, UK
        {\tt\small ag153@le.ac.uk}}%
\thanks{$^{2}$Ivan Yu. Tyukin is with the University of Leicester, Department of Mathematics, Leicester, LE1 7RH, UK
        and with Saint-Petersburg State Electrotechnical University, Department of Automation and Control Processes, Prof. Popova str. 5, 197376, Russia {\tt\small I.Tyukin@le.ac.uk}  }%
\thanks{$^{3}$Henk Nijmeijer is with Eindhoven University of Technology, Department of Mechanical Engineering, P.O. Box 513 5600 MB, Eindhoven, The Netherlands {\tt\small h.nijmeijer@tue.nl}  }%
}

\begin{document}

\maketitle
\thispagestyle{empty}
\pagestyle{empty}

\begin{abstract}

We provide several characterizations of convergence to unstable equilibria in nonlinear systems. Our current contribution is three-fold. First we present
 simple algebraic conditions for establishing local convergence of non-trivial solutions of nonlinear systems to unstable equilibria. The conditions are based on the earlier work \cite{Tyukin:SIAM:2013} and can be viewed as an extension of the Lyapunov's first method in that they apply to systems in which the corresponding Jacobian has one zero eigenvalue. Second, we show that for a relevant subclass of systems, persistency of excitation of a function of time in the right-hand side of the equations governing dynamics of the system ensure existence of an attractor basin such that solutions passing through this basin in forward time converge to the origin exponentially. Finally we demonstrate that conditions developed in  \cite{Tyukin:SIAM:2013} may be remarkably tight.
%
\end{abstract}


\section{Introduction}

Analysis of asymptotic behavior of solutions of nonlinear systems is one of the central pillars of modern control theory. Lyapunov stability \cite{Lyapunov:1892} is an example of such characterizations. The notion of Lyapunov stability and analysis methods that are
based on this notion are proven successful in a wide range of
engineering applications (see e.g. \cite{Nijmeijer_90},
\cite{Isidory}, \cite{Ljung_99}, \cite{Narendra89} is a
non-exhaustive list of references).
The popularity and success of the concept of Lyapunov stability
resides, to a substantial degree, in the convenience and
utility of the method of Lyapunov functions for
assessing asymptotic properties of solutions of ordinary
differential equations. Instead of deriving the solutions
explicitly it suffices to solve an algebraic inequality
involving partial derivatives of a given Lyapunov candidate
function. Yet, as the methods of control expand from purely engineering
applications into a wider area of science, there is a need for
maintaining behavior that fails to obey the usual requirement of
Lyapunov stability.

There are numerous examples of systems
possessing Lyapunov-unstable, yet attracting, invariant sets
\cite{Andronov:1973},  e.g., in the domains of aircraft dynamics and design
of synchronous generators \cite{Bautin:1990} (pp. 313--356). Other examples include
 models of decision-making sequences \cite{Rabinovic_2006}, \cite{Rabinovic_2008}, \cite{NN_template_matching},
flutter suppressors \cite{GOMAN}, the general problem of universal adaptive stabilization \cite{Dyn_Con:Ilchman:97,CDC:Pomet:1992}, and
 problems of adaptive observer design for systems with nonlinear in parameter right-hand side \cite{Tyukin:Automatica:2013}.
Finding rigorous, convenient and at the same time tight criteria for asymptotic
convergence to Lyapunov-unstable invariant sets, however, is a
non-trivial problem.


Criteria
for checking attractivity of unstable point attractors in a
rather general setting have been proposed in
\cite{Rantzer}, and were further developed in
\cite{Masubuchi:2007,Vadia:2008}. These results apply to
systems in which almost all points in a neighborhood of the
attractor correspond to solutions converging to the attractor
asymptotically. However, as Figure \ref{fig:domains}
\begin{figure*}
\begin{center}
\begin{minipage}[h]{0.17\linewidth}
\begin{center}
\includegraphics[width=\textwidth]{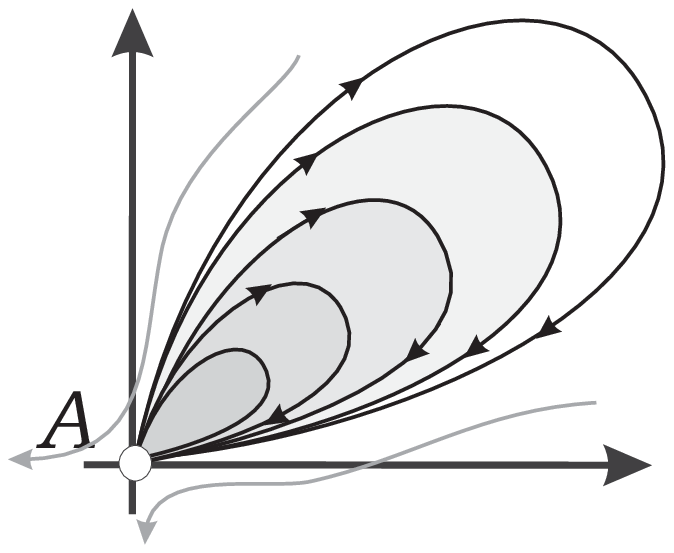}
\vspace{5mm}
a
\end{center}
\end{minipage}
\hspace{5mm}
\begin{minipage}[h]{0.17\linewidth}
\begin{center}
\includegraphics[width=\textwidth]{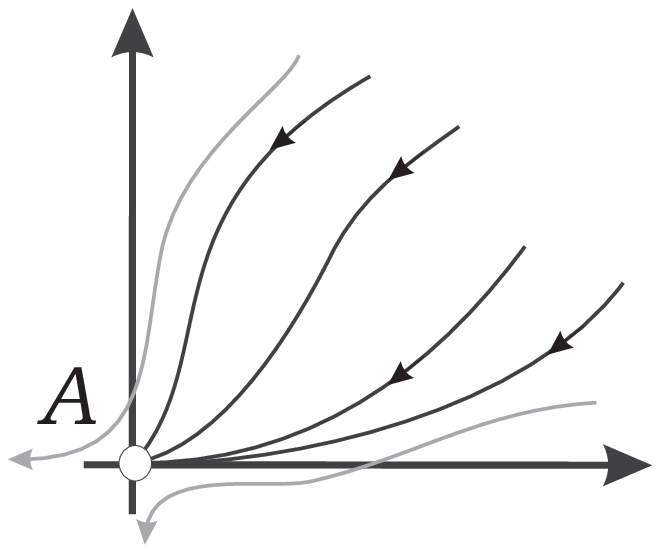}
\vspace{5mm}
b
\end{center}
\end{minipage}
\hspace{5mm}
\begin{minipage}[h]{0.17\linewidth}
\begin{center}
\includegraphics[width=\textwidth]{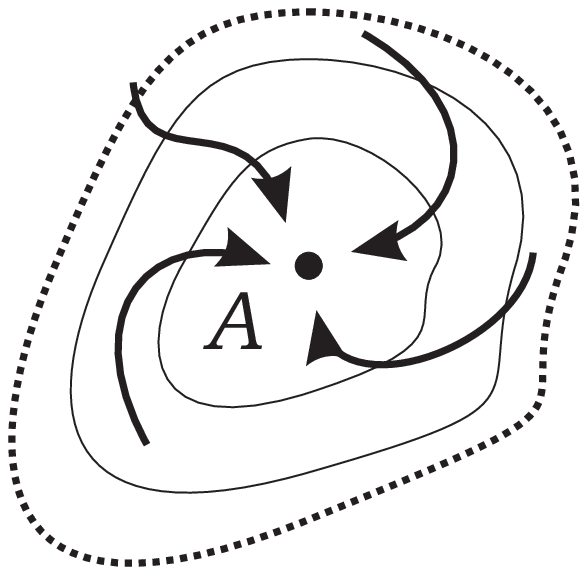}
\vspace{5mm}
c
\end{center}
\end{minipage}
\hspace{5mm}
\begin{minipage}[h]{0.22\linewidth}
\begin{center}
\includegraphics[width=\textwidth]{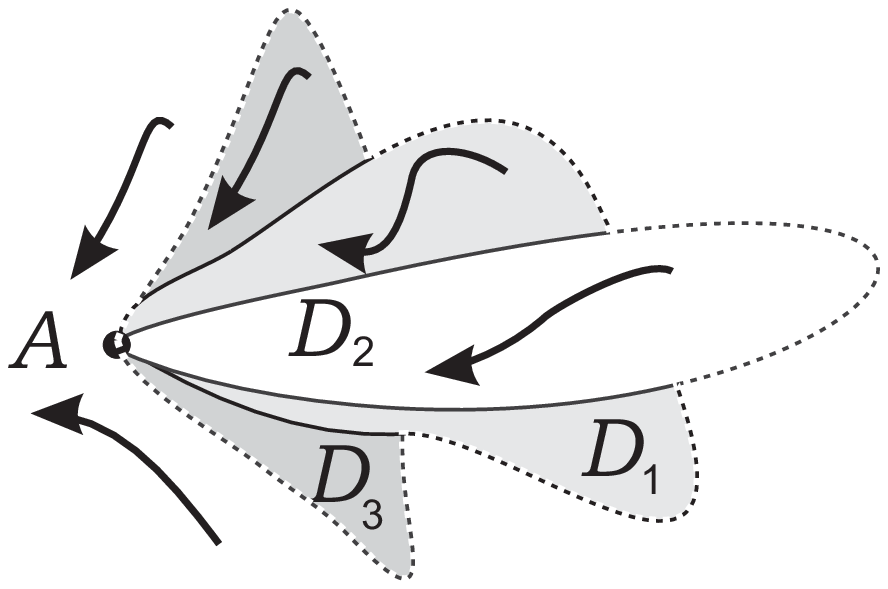}
\vspace{5mm}
d
\end{center}
\end{minipage}
\end{center}
\caption{Panels a,b: a diagrammatic picture of an elliptic sector (panel a) and a parabolic sector (panel b) in the phase space of nonlinear systems \cite{Bautin:1990}. Panel c: a system of neighborhoods corresponding to the level sets of a Lyapunov function of a {\it stable} set $A$. Every neighborhood is forward invariant; all trajectories are contained in the largest neighborhood (marked by dashed line).  Panel d: a system of forward invariant sets of an {\it attracting} set $A$. Trajectories passing through the union of these sets (marked by dashed line) remain there in forward time. }\label{fig:domains}
\end{figure*}
illustrates (panels a,b), there are alternatives that do not comply with these assumptions. On
the other hand techniques which can be used to
address the questions above for equations (\ref{eq:problem}),
such as, e.g.,  \cite{SIAM_non_uniform_attractivity}, lack the convenience of
the method of Lyapunov functions.

Recently,  an approach has been proposed in \cite{Tyukin:SIAM:2013} that enables to extend the method of Lyapunov functions to a class of systems with unstable invariant sets. Standard Lyapunov stability of a set is equivalent to existence of a nested family of neighborhoods, containing this set, which are forward-invariant with respect to dynamics of the system (Fig. \ref{fig:domains}, panel c). And the method of Lyapunov functions is a tool for finding such nested family of neighborhoods. In the approach proposed in \cite{Tyukin:SIAM:2013} families of nested neighborhoods are replaced with  collections of forward invariant sets containing the set of interest (Fig. \ref{fig:domains}, panel d). These sets are not necessarily neighborhoods of the invariant set. Yet, if they exist and at least one of such sets has a non-zero measure, then the original invariant set is clearly weakly attracting in Milnor's sense \cite{Milnor_1985}.

 Utility of proposed in \cite{Tyukin:SIAM:2013} criteria specifying sets of forward-invariance, as well as their non-invariant counterparts, is that the resulting conditions are akin to the ones used in the method of Lyapunov functions and other tangential conditions \cite{SICON:2011:Aubin}, \cite{Nagumo:1942} (see also \cite{Chetaev:1961} for conditions of instability). This offers obvious advantage. On the other hand, few questions remain regarding the approach in \cite{Tyukin:SIAM:2013}, including a) existence of a similar analogue of Lyapunov's first method, b) possibility of unstable yet exponential convergence and corresponding conditions, and c) how tight the derived conditions for boundedness and asymptotic convergence may be?  Answering to these questions is the main goal of this work.

 The paper is organized as follows. Main notational agreements and conventions are provided in Section \ref{sec:notation}, Section \ref{sec:preliminaries} presents general class of systems considered in \cite{Tyukin:SIAM:2013}, main assumptions and one illustrative theoretical result. In Section \ref{sec:main} we show how these results can be used to extend the first method of Lyapunov. Furthermore, for a subclass of systems relevant in problems of adaptive observers design, we provide conditions ensuring that not only that the concerned equilibrium is a weak attractor but also that the convergence to this attractor is exponential. Finally, we demonstrate that boundedness and attractivity conditions that the approach from  \cite{Tyukin:SIAM:2013} provides may sometimes be necessary too. Section \ref{sec:conclusion} concludes the paper.

\section{Notation}\label{sec:notation}

The following notational conventions are used throughout the paper:
\begin{itemize}
\item $\Real$ denotes the set of real numbers, $\Real_{>a} = \{x\in\Real\ |\ x>a\}$, and  $\Real_{\geq a} = \{x\in\Real\ |\ x\geq a\}$;
\item the Euclidean norm of $x \in \Real^n$ is denoted by $\norm{x}$, $\norm{x}^2 = x^T x$, where ${}^T$ stands for transposition;
      \item the space of $n\times n$ matrices with real entries is denoted by $\Real^{n\times n}$; let  $P\in\Real^{n\times n}$, then $P>0$ ($P\geq 0$) indicates that $P$ is symmetric and positive (semi-)definite; $I_{n}$ denotes  the $n\times n$ identity matrix;
    \item let $\Gamma\in\Real^{n\times n}$, $\Gamma>0$, and $x\in\Real$, then $\|x\|_{\Gamma^{-1}}^2$ denotes $x^{T}\Gamma^{-1}x$;
    \item  by ${L}^n_\infty[t_0,T]$, $t_0\in\Real$, $T\in\Real, \ T\geq t_0$ we denote the space of all functions $xf:[t_0,T]\rightarrow\Real^n$ such that $\|f\|_{\infty,[t_0,T]}=\ess \sup\{\|f(t)\|,t \in [t_0,T]\}<\infty$; $\|f\|_{\infty,[t_0,T]}$ stands for the ${L}^n_\infty[t_0,T]$ norm of $f(\cdot)$; if the function $f$ is defined  on a set larger than $[t_0,T]$ then notation $\|f\|_{\infty,[t_0,T]}$ applies to the restriction of $f$ on $[t_0,T]$;
    \item $\mathcal{C}^{r}$ denotes the space of continuous functions that are at least $r$ times differentiable;
    \item the symbol $\mathcal{K}_0$ denotes   the set of all non-decreasing
continuous functions $\kappa: \Real_{\geq 0}\rightarrow
\Real_{\geq 0}$ such that $\kappa(0)=0$; $\mathcal{K} \subset \mathcal{K}_0$
is the subset of strictly increasing functions, and $\mathcal{K}_\infty \subset \mathcal{K}$
consists of functions from $\mathcal{K}$ with infinite limit: $\lim_{s\rightarrow\infty}\kappa(s)=\infty$.
\end{itemize}

\section{Preliminaries}\label{sec:preliminaries}

Consider system (\ref{eq:problem})
\begin{equation}\label{eq:problem}
\begin{split}
\dot{x}&=f(x,\lambda,t),\\
\dot{\lambda}&=g(x,\lambda,t),
\end{split}
\end{equation}
where the vector-fields $f: \
\Real^n\times\Real\times\Real\rightarrow\Real^n$, $g: \
\Real^n\times\Real\times\Real\rightarrow\Real$ are continuous
and locally Lipschitz w.r.t. $x$, $\lambda$ uniformly in $t$.
The point $x=0,
\lambda=0$ is assumed to be an equilibrium of (\ref{eq:problem}).

Let  $\mathcal{D}$ be an open subset of $\Real^n$  and
${\Lambda}=[c_1,c_2], \ c_1\leq 0, \ c_2 > 0$, be
an interval. Suppose that the closure
$\overline{\mathcal{D}}$ of $\mathcal{D}$ contains the
origin, and denote
$\mathcal{D}_{\Omega}=\overline{\mathcal{D}}\times{\Lambda}\times\Real$.
Finally,  we suppose that the right-hand side of
(\ref{eq:problem}) satisfies  Assumptions \ref{assume:stable},
\ref{assume:unstable} below.
\begin{assume}\label{assume:stable}There exists a function $V:\Real^n\rightarrow \Real_{\geq 0}$,  $V\in\mathcal{C}^{0}$, differentiable everywhere except possibly at the origin, and five functions of one variable, $\underline{\alpha},\bar{\alpha}\in\mathcal{K}_{\infty}$, $\alpha:\Real_{\geq 0}\rightarrow\Real$,
$\alpha\in\mathcal{C}^{0}([0,\infty))$, $\alpha(0)=0$, $\beta: \ \Real_{\geq 0}\rightarrow\Real_{\geq 0}$, $\beta\in\mathcal{C}^{0}([0,\infty))$, $\varphi\in\mathcal{K}_0$
such that  for every   $(x,\lambda,t)\in{(\overline{\mathcal{D}}\setminus\{0\})\times\Lambda\times\Real}$  the following properties hold:
\begin{equation}\label{eq:stable}
\begin{split}
\underline{\alpha}(\|x\|)\leq V(x) \leq \bar{\alpha}(\|x\|),\\
\frac{\pd V}{\pd x}f(x,\lambda,t)\leq  \alpha(V(x)) + \beta(V(x))\varphi(|\lambda|).
 \end{split}
\end{equation}
\end{assume}
\begin{assume}\label{assume:unstable}  There exist functions $\delta, \xi \in\mathcal{K}_0$ such that the following inequality holds for all $(x,\lambda,t)\in{\mathcal{D}_{\Omega}}$:
\begin{equation}\label{eq:g:Lipschitz:monotone:1}
\begin{split}
- \xi (|\lambda|) - \delta(\|x\|) &\leq  g(x,\lambda,t)\leq 0.
\end{split}
\end{equation}
\end{assume}

The following is an example of a Lyapunov-like condition for establishing whether the origin of (\ref{eq:problem}).
\begin{cor}[\cite{Tyukin:SIAM:2013}]\label{cor:boundedness_unstable} Consider  system (\ref{eq:problem}), and let $\mathcal{D}=\Real^n$, $\Lambda=\Real$. Suppose that Assumptions \ref{assume:stable}, \ref{assume:unstable} hold, there exists  a function
$\psi: \ \psi\in\mathcal{K}\cap\mathcal{C}^{1}((0,\infty))$ and a
positive constant $a\in\Real_{>0}$ such that for all $V\in(0,a]$
\begin{equation}\label{eq:positive:invariance:attracting}
\begin{split}
\frac{\pd \psi(V)}{\pd V}\left[\alpha(V)+\beta(V)\varphi(\psi(V))\right]\\
+\delta\left(\underline{\alpha}^{-1}(V)\right)+\xi\left(\psi(V)\right)\leq 0,
\end{split}
\end{equation}

Then
\begin{itemize}
\item[(a) ] the set
\begin{equation}\label{eq:positive:invariance:domain:cor}
\begin{split}
\Omega_a=&\{(x,\lambda) \ | \  x\in\Real^n, \ \lambda\in\Real_{\geq 0},\\
 & \psi(a)\geq\lambda\geq\psi(V(x)), \ V(x)\in[0,a]\}
\end{split}
\end{equation}
is forward  invariant.
 \end{itemize}
 Furthermore, for every solution of (\ref{eq:problem}) starting in $\Omega_a$
\begin{itemize}
\item[(b) ] there exists a limit
\begin{equation}\label{eq:positive:invariance:limit:1}
  \lim_{t\rightarrow\infty}\lambda(t)=\lambda', \;\; \lambda'\in [0,\psi(a)] .  \nonumber
\end{equation}
\item[(c)] If, in addition, the function $g(x,\lambda,
    \cdot )$ is uniformly continuous then:
\begin{equation}\label{eq:positive:invariance:limit:1.5}
\lim_{t\rightarrow\infty}g(x(t),\lambda',t)=0.  \nonumber
\end{equation}
\end{itemize}
\end{cor}

In the next section we show how this and other results from \cite{Tyukin:SIAM:2013} can be used in extending classical Lyapunov's first method. Furthermore, we will establish conditions ensuring exponential convergence of the solutions to the origin and show that sometimes these results may enable to derive necessary and sufficient conditions for existence of weak attractors in (\ref{eq:problem}).

\section{Main Results}\label{sec:main}

\subsection{Extension of Lyapunov's First Method}

Consider system
\begin{equation}\label{eq:system}
\begin{split}
\dot{z}&=p(z)
\end{split}
\end{equation}
where the vector-field $p: \
\Real^m\rightarrow\Real^m$ is continuous
and locally Lipschitz w.r.t. $z$. Furthermore, let it be differentiable at the origin, $p(0)=0$, and
\[
J=\frac{\pd p}{\pd z} (0)
\]
be the corresponding Jacobian matrix. Finally, let
\[
\sigma_1,\sigma_2,\dots ,\sigma_m
\]
be the eigenvalues of $J$ with $\sigma_1=0$ and real parts of all other eigenvalues be negative. In what follows we are interested in finding a set of simple Lyapunov-like conditions that would enable us to establish whether the origin is a weak attractor or not.

Without loss of generality, consider dynamics of (\ref{eq:system}) in the coordinates
\[
\left(\begin{array}{c}
       x \\
       \lambda
      \end{array}\right)=T z,
\]
$x\in\Real^n$, $\lambda\in\Real$, $m=n+1$, where the $m\times m$ non-singular matrix $T$ is such that the structure of $T J T^{-1}$ is as follows
\[
T J T^{-1}=\left(\begin{array}{cc}A & b\\ 0 & 0\end{array}\right), \ A\in\Real^{n\times n}, \ b\in\Real^{n}.
\]
It is clear that such matrix $T$ will always exist and that  $\sigma_2,\dots,\sigma_m$ are the eigenvalues of $A$.

Dynamics of (\ref{eq:system}) in the coordinates $(x,\lambda)$ is
\begin{equation}\label{eq:system:2}
\begin{split}
\dot{x}&=f(x,\lambda)\\
\dot{\lambda}&=g(x,\lambda),
\end{split}
\end{equation}
with
\[
\begin{split}
\frac{\pd f}{\pd x}(0,0)=A, \ \frac{\pd f}{\pd \lambda}(0,0)=b, \\
\frac{\pd g}{\pd x}(0,0)=0, \ \frac{\pd g}{\pd \lambda}(0,0)=0.
\end{split}
\]
The following can now be formulated for (\ref{eq:system:2})

\begin{thm} Consider system (\ref{eq:system:2}), and let the function $g$ be differentiable at least twice. Let $G$ be the Hessian of $g$ at the origin and
let $G$ be sign-definite.

Then $(0,0)$ is a weak (Milnor) attractor for (\ref{eq:system:2}).
\end{thm}

\begin{proof} Consider dynamics of (\ref{eq:system:2}) in a vicinity of the origin:
\[
\begin{split}
\dot{x}&=Ax + b\lambda + o(\|(x,\lambda)\|)\\
\dot{\lambda}&= \frac{1}{2}(x,\lambda)^{T}G(x,\lambda)(x,\lambda)+o(\|(x,\lambda)\|^2).
\end{split}
\]
Since real parts of the eigenvalues of $A$ are negative, there are symmetric positive-definite matrices $H$, $Q$ such that
\[
HA + A^TH \leq - Q.
\]
Without loss of generality suppose that $G<0$ (if $G>0$ we can replace $\lambda$ with $\tilde\lambda=-\lambda$ to get a system representation as in (\ref{eq:system:2}) but with $G<0$).

Let $V=x^{T}Hx$ and consider the function $\psi(V)=k \sqrt{V}$, $k>0$. It is clear that there exist $\alpha>0$, $\beta>0$, and $c_1>0$, independent on $k$, such that
\[
\dot{V}\leq -\alpha V  + \beta \sqrt{V}|\lambda| + c_1 \sqrt{V}|o(\|(x,\lambda)\|)|.
\]
Furthermore, for any given $\varepsilon>0$ there is a neighborhood $\Omega_1$ of the origin:
\[
|o(\|(x,\lambda)\|)| \leq  \varepsilon (\sqrt{V(x)}+|\lambda|) \ \mathrm{for \ all} \ (x,\lambda)\in\Omega_1.
\]
Hence there exists a neighborhood $\Omega_2$ of the origin and  $\alpha_1>0$, $\beta_1>0$, independent on $k$:
\[
\dot{V}\leq -\alpha_1 V(x)  + \beta_1 \sqrt{V(x)}|\lambda| \ \mathrm{for \ all} \ (x,\lambda)\in\Omega_2.
\]
Similarly, there is a neighborhood $\Omega_3$ containing the origin and a constant $\gamma>0$ such that
\[
- \gamma (V(x) + \lambda^2) \leq \dot{\lambda}\leq 0 \ \mathrm{for \ all} \ (x,\lambda)\in\Omega_3.
\]

According to \cite{Tyukin:SIAM:2013} (Lemma 1) existence of an interval $(0,a]$, $a>0$ such that
\begin{equation}\label{eq:boundedness:cond}
\begin{split}
\frac{\pd \psi}{\pd V} (-\alpha_1 V  + \beta_1 \sqrt{V} \psi(V)) +\gamma V + \gamma\psi^2(V) \leq 0
\end{split}
\end{equation}
would imply that the intersection of the corresponding sets $\Omega_a$ and $\Omega_2\cap \Omega_3$ is forward-invariant. Substituting $\psi(V)=k\sqrt{V}$ into the left-hand side of the expression above one obtains
\[
\begin{split}
&-k\frac{\alpha_1}{2}\sqrt{V} + k^2 \frac{\beta_1\sqrt{V}}{2} + \gamma (V+k^2 V)=\\
&\sqrt{V}(-k\frac{\alpha_1}{2}+k^2\frac{\beta_1}{2}+\gamma (1+k^2)\sqrt{V}).
\end{split}
\]
It is therefore clear that  (\ref{eq:boundedness:cond}) holds for all
\[
{V}\leq a =  \left[\left(k\frac{\alpha_1}{2}-k^2\frac{\beta_1}{2}\right)/(\gamma (1+k^2))\right]^2.
\]
The measure of the set $\Omega_a$ ($a\neq 0$) is non-zero. Moreover, since $G$ is sign-definite,
\[
\lim_{t\rightarrow\infty}x(t)=0, \ \lim_{t\rightarrow\infty}\lambda(t)=0.
\]
Hence $(0,0)$ is a weak attractor.
\end{proof}

\subsection{Exponential Convergence}

Consider a subclass of (\ref{eq:problem})
that is relevant in the problem of adaptive observer design \cite{Tyukin:Automatica:2013}:
\begin{equation}\label{eq:system:observer}
\begin{split}
\left(\begin{array}{c}
        \dot{x}_1\\
        \dot{x}_2
        \end{array}
\right)
        &= \left(\begin{array}{cc} A & b\varphi^T(t)\\
                                  -\Gamma \varphi(t)C^T & 0
                                  \end{array}\right)  \left(\begin{array}{c}
        {x}_1\\
        {x}_2
        \end{array} \right) \\
        & + \left(\begin{array}{c}b\\ 0\end{array}\right)v(t,\lambda)\\
\dot{\lambda}&=-\gamma |C^T x_1|, \\
& x=(x_1,x_2), \ x(t_0)=x_0, \ \lambda(t_0)=\lambda_0,
\end{split}
\end{equation}
where $x_1\in\Real^k$, $x_2\in\Real^m$, $\lambda\in\Real$ are state variables ($k+m=n$), $x_0\in\Real^n$, $\lambda_0>0$ are initial conditions, $\Gamma\in\Real^{m\times m}$ is a positive-definite symmetric matrix; matrix $\bfA\in\Real^{k\times k}$, and vectors $\bfB$, $\bfC$ are
supposed to
satisfy
\begin{equation}\label{eq:condition_MKY}
         \left\{\begin{split}
         &\bfP \bfA + \bfA^T \bfP\leq - \bfQ \\
         &\bfP\bfB = \bfC.
         \end{split}\right.
\end{equation}
for some  symmetric positive definite
matrices $\bfP$, $\bfQ$. Functions $\varphi:\Real\rightarrow\Real^{m}$, $v:\Real\times\Real\rightarrow\Real$ are continuous and differentiable with bounded derivatives. Furthermore, $v(t,0)=0$ for all $t$.

It has been shown in  \cite{Tyukin:Automatica:2013}  that if the function $\varphi$ is persistently exciting, bounded and with bounded derivative then  there exists an interval $(0,\gamma^\ast(x_0,\lambda_0)]$, $\gamma^\ast(x_0,\lambda_0)>0$, such that for all constant $\gamma$ taken from this interval solutions of (\ref{eq:system:observer}) are bounded and
$0<\lambda(t)<\lambda_0 \ \forall \  t\geq t_0$.
Furthermore, $\lim_{t\rightarrow \infty}\lambda(t)=0$ and $\lim_{t\rightarrow\infty}x(t)=0$. One can arrive at the same conclusion using e.g. Corollary 1 or Lemma 1 from \cite{Tyukin:SIAM:2013}. The question, however, is if such convergence can be made exponential. An answer to this question is provided below.

\begin{thm} Consider system (\ref{eq:system:observer}) with $A,b,C$ satisfying (\ref{eq:condition_MKY}). Let us suppose that parameter $\gamma$  and initial conditions are chosen so that solutions of (\ref{eq:system:observer}) are bounded. Furthermore, let the vector
\[
\phivec(t)=\left(\varphi(t),\left.\int_{0}^{1} \frac{\pd v(t,a)}{\pd a}\right|_{a=\lambda(t) s } ds\right)
\]
be persistently exciting
\[
\exists \ T,\mu>0: \   \int_{t}^{t+T}\phi(\tau)\phi(\tau)^{T}d\tau\geq \mu I_{m+1} \ \forall t\geq t_0,
\]
and there is an $M_\phi>0$: $\max\{\|{\phivec}(t)\|,\|\dot{\phivec}(t)\|\}<M_{\phi}$ for all $t\geq t_0$. Then $x(t)$, $\lambda(t)$ converge to the origin exponentially fast.
\end{thm}
\begin{proof} The proof is organized as follows. First, we demonstrate that for all $t,t_0$, $t\geq t_0$ (for which the solutions are defined) the vectors $x_1(t)$ and $q(t)=(x_2(t),\lambda(t))$ satisfy the following conditions:
\begin{itemize}
\item[C1)] $\exists \ c_1: \ \|x_1\|_{\infty,[t,\infty)}\leq c_1 (\|x_1(t)\|+\|q(t)\|) \ \forall \  t\geq t_0$

\item[C2)] $\exists \ c_2: \ \|q\|_{\infty,[t,\infty)}\leq c_2 (\|x_1(t)\|+\|q(t)\|)\ \forall \  t\geq t_0$

\item[C3)] $\exists \ c_3: \ \|x_1\|_{2,[t,\infty)}\leq c_3 (\|x_1(t)\|+\|q(t)\|) \ \forall \  t\geq t_0$,
\end{itemize}
where $c_1$,$c_2$, and $c_3$ are independent of $t$. Second, we show that
\begin{itemize}
\item[C4)] $\exists \ c_4: \ \|q\|_{2,[t,\infty)}\leq c_4 (\|x_1(t)\|+\|q(t)\|)$.
\end{itemize}
Finally, we invoke a result from \cite{Tyukin:2011},\cite{Lorea_2002} to show that C1--C4 imply exponential convergence of the observer\footnote{Lemma \ref{lem:exp} is a minor modification of Lemma 3 in \cite{Lorea_2002} in which condition (\ref{eq:norm_bounds}) is no longer required to hold in a neighborhood of the origin. Instead we assume that (\ref{eq:norm_bounds}) is satisfied along a given solution. The proof of the modified statement is identical to the one presented in \cite{Lorea_2002}. The  part of the statement of Lemma 3 in \cite{Lorea_2002} concerning local and global exponential Lyapunov stability of the origin, however, is no longer applicable to the special case considered here.}.
\begin{lem}\label{lem:exp} Let $x:\Real\rightarrow\Real^n$ be a function satisfying
\begin{equation}\label{eq:norm_bounds}
\max\{\|x\|_{2,[t,\infty)},\|x\|_{\infty,[t,\infty)}\}\leq c\|x(t)\|, \ \forall \ t\geq t_0.
\end{equation}
Then
\[
\|x(t)\|\leq ce^{1/2}e^{-\frac{t-t_1}{2c^2}}\|x(t_1)\|, \ \forall \ t\geq t_1\geq t_0.
\]
\end{lem}

{\it First part.}  We have that for all $t\geq t_1\geq t_0$
\[
0\leq\lambda(t)\leq \lambda(t_1)-\gamma \int_{t_1}^t|\bfC^{T}x_1(\tau,x_0)|d\tau.
\]
Let $P>0$ be a matrix satisfying (\ref{eq:condition_MKY}). Consider the following function
\[
V=x_1^{T} P x_1 + \|x_2\|^2_{\Gamma^{-1}}+\frac{D_v}{\gamma} \lambda^2,
\]
where $D_v$ is $\max_{t\geq t_0, \ \lambda\in[0,\lambda_0]} |\pd v(t,\lambda)/\pd \lambda|$. It is clear that
\begin{equation}\label{eq:dot_V:0}
\begin{split}
\dot{V}&\leq - x_1^{T} Q x_1+2 x_1^{T} P b \varphi^T(t)x_2 - 2x_2^{T}\bfC^{T} x_1 \varphi(t) \\
&+  2 x_1^{T} P b v(t,\lambda) - 2 D_v \lambda |\bfC^{T}x_1|\leq - x_1^{T} Q x_1\\
&-2 |\bfC^{T}x_1|(D_v\lambda-\mathrm{sign}(\bfC^{T}x_1)v(t,\lambda)).
\end{split}
\end{equation}
Noticing that $\lambda(t)\geq 0$,
\[
v(t,\lambda)=v(t,\lambda)-v(t,0)=\int_{0}^{1} \left. \frac{\pd v(t,a)}{\pd a}\right|_{a=s\lambda}ds \lambda,
\]
and consequently,
\[
|\mathrm{sign}(\bfC^{T}x_1)v(t,\lambda)|\leq D_v \lambda,
\]
we can conclude that the term $D_v\lambda-\mathrm{sign}(\bfC^{T}x_1)v(t,\lambda)$ in (\ref{eq:dot_V:0}) is always non-negative, and hence
\begin{equation}\label{eq:dot_V}
\dot{V}\leq  - x_1^{T} Q x_1.
\end{equation}
Thus C1, C2 hold. Furthermore, in view of (\ref{eq:dot_V}), one can derive that C3 holds as well.

{\it Second part.} Let us show that C4 holds. In order to do so we use the method described in \cite{Lorea_2002}. Consider the variable
\[
z=q -  \phi(t) b^{T}x_1
\]
and calculate its derivative:
\[
\dot{z}=\dot{q}-\dot{\phi}(t)b^{T}x_1 - {\phi}(t)b^{T}A x_1 - {\phi}(t)b^{T}b(\varphi(t)^{T}x_2 + v(t,\lambda)).
\]
Noticing that $v(t,\lambda)=\int_{0}^{1} \left. \frac{\pd v(t,a)}{\pd a}\right|_{a=s\lambda}ds \lambda$ we conclude that
\[
\varphi(t)^{T}x_2 + v(t,\lambda)=\phi(t)^{T}(x_2,\lambda)=\phi(t)^{T}q.
\]
Hence
\[
\begin{split}
\dot{z}=& \dot{q}-\dot{\phi}(t)b^{T}x_1 - {\phi}(t)b^{T}A x_1 - {\phi}(t)b^{T}b \phi(t)^{T}z\\
&-\phi(t)b^{T} b \phi(t)^{T} \phi(t)^T b x_1\\
=& - \phi(t)b^{T}b \phi(t)^{T} z + \chi(t,x_1),
\end{split}
\]
where
\[
\begin{split}
\chi(t,x_1)=&\dot{q}-\dot{\phi}(t) b^{T}x_1 -\\
& \phi(t)b^{T} \bfA x_1-\phi(t)b^{T} b \phi(t)^{T} \phi(t)^T b x_1.
\end{split}
\]
Recall that $\dot{q}=\mathrm{col}(-\Gamma \varphi(t)C^T x_1,-\gamma |C^{T}x_1|)$ and  that $\phi(t),\dot{\phi}(t)$ are bounded. Hence there is a constant
 $M>0$ such that
\[
\|\chi(t,x_1)\|\leq M \|x_1\|.
\]
Taking into account that
$\phi(t)$ is persistently exciting we can conclude that there are constants
$\beta_1,\beta_2$:
\[
\|z\|_{2,[t,\infty)}\leq \beta_1\|z(t)\|+ \beta_2 \|x_1\|_{2,[t,\infty)}.
\]
for all $t\geq t_0$. Given that  $\|\phi(t)
b^{T}x_1\|\leq
M_\phi\|\bfB\|\|x_1\|$ we obtain:
\[
\begin{split}
\|q\|_{2,[t,\infty)}\leq& \beta_1(\|q(t)\|+M_\phi \|\bfB\|\|x_1(t)\|) \\
&+  (M_\phi\|\bfB\|+\beta_2) \|x_1\|_{2,[t,\infty)}.
\end{split}
\]
Hence C4 holds.

{\it Third part}  follows from Lemma \ref{lem:exp}.
\end{proof}

\subsection{How tight are the estimates of regions of forward invariance?}\label{subsec:tight}

On the one hand, since Assumptions \ref{assume:stable} and
\ref{assume:unstable} are inherently conservative, our results
bear a degree of conservatism. On the other hand, if viewed as
conditions for the mere existence of (weakly) attracting sets, they
can sometimes be remarkably precise. This is illustrated with
the example below.

\begin{example}\label{example:tight}\normalfont  Consider  system
\begin{equation}\label{eq:small_gain:example:1}
\begin{split}
\dot{x}_1&=-\tau_1 x_1 + c_1 x_2\\
\dot{x}_2&=-c_2|x_1|, \ \tau_1, \ c_1, \ c_2\in\Real_{>0},
\end{split}
\end{equation}
and let us determine the values of $c_1,c_2$, and $\tau_1$ such
that the origin is a weak attractor for
(\ref{eq:small_gain:example:1}).  We will do this by invoking
Corollary \ref{cor:boundedness_unstable}. It is clear that
solutions of (\ref{eq:small_gain:example:1}) are defined for
all $t$.  Thus letting $\mathcal{D}_x=\Real$,
$\mathcal{D}_\lambda=\Real$ we can easily see that Assumption
\ref{assume:stable} holds for the first equation with
$V(x_1)=x_1^2$:
\[
\dot V\leq-2 \tau_1 V + 2 c_1 \sqrt{V}|x_2|,
\]
and Assumption \ref{assume:unstable} is satisfied for the
second equation with $\delta(|x_1|)=c_2 |x_1|=c_2\sqrt{V}$,
$\xi(|x_2|)=0$. Let us pick
\[
\psi(V)=p\sqrt{V}, \ p\in\Real_{>0},
\]
and consider
\begin{equation}\label{eq:example:discussion}
\begin{split}
&\frac{\pd \psi}{\pd V}(-2\tau_1 V + 2c_1\sqrt{V}\psi(V))\\
& \ \ \ \  +c_2\sqrt{V}=(-p\tau_1+c_1p^2 +c_2)\sqrt{V}.
\end{split}
\end{equation}
The expression above is defined for $V\in(0,\infty)$, and it is
non-positive for $c_2\leq p(\tau_1-p c_1)$.  The right-hand
side of the last inequality is maximal at $p=\tau_1/(2 c_1)$.
Thus we can conclude that
\begin{equation}\label{eq:example:discussion:1}
c_2\leq \frac{\tau_1^2}{4 c_1}
\end{equation}
ensures that (\ref{eq:example:discussion}) is non-positive for
all $V\in\Real_{>0}$. Hence, according to Corollary
\ref{cor:boundedness_unstable}, the set
\[
\Omega=\{(x_1,x_2) \ | x\in\Real, \ x_2\in \Real_{>0}, \ x_2\geq  \frac{\tau_1}{2 c_1}|x_1|\}
\]
is forward-invariant.  Moreover, solutions of
(\ref{eq:small_gain:example:1}) starting in $\Omega$ are
bounded and satisfy $\lim_{t\rightarrow\infty}x_1(t)=0$. Given
that  $x_2(t)$  is bounded, Barbalatt's lemma (applied to the
first equation) implies that
$\lim_{t\rightarrow\infty}x_2(t)=0$. Hence, the origin of
system (\ref{eq:small_gain:example:1}) satisfying condition
(\ref{eq:example:discussion:1}) is a weak attractor.

Let us now see if the attractor  persists when inequality
(\ref{eq:example:discussion:1}) does not hold. Consider an
auxiliary system:
\begin{equation}\label{eq:example:discussion:2}
\begin{split} \dot{x}_1&=-\tau_1 x_1 + c_1
x_2\\ \dot{x}_2&=-c_2 x_1, \ \tau_1, \ c_1, \ c_2\in\Real_{>0},
\end{split}
\end{equation}
and let $(x_1(t),x_2(t))$ be a non-trivial solution of
(\ref{eq:example:discussion:2}). If the roots of
$\chi(s)=s^2+\tau_1 s+c_1 c_2$ have non-zero imaginary real
parts then the sign of $x_2(t)$ will necessarily alternate.
This, however, implies that no non-trivial solutions of
(\ref{eq:small_gain:example:1}) converge to the origin. The
roots of $\chi(s)$ are real, however, only if
(\ref{eq:example:discussion:1}) holds. Therefore, in this
particular case, condition (\ref{eq:example:discussion:1}) is
not only sufficient but it is also necessary for the origin of
(\ref{eq:small_gain:example:1}) to be an attractor.

Phase curves of (\ref{eq:small_gain:example:1}) illustrating
this point are provided in Fig. \ref{fig:discussion}. The left plot
corresponds to the case when condition
(\ref{eq:example:discussion:1}) is satisfied. As we can see,
solutions are asymptotically approaching the origin
(marked by a black circle).  The right plot shows phase curves of
the system in which the value of $c_2$ is larger than
$\tau_1^2/(4c_1)$. In this case, as can be clearly seen  in the
inset at the bottom right corner, solutions  of the system do
not approach the origin asymptotically. They are lingering in
its neighborhood for a while, and then eventually escape.
\begin{figure*}[!ht]
\centering
\includegraphics[width=0.8\textwidth]{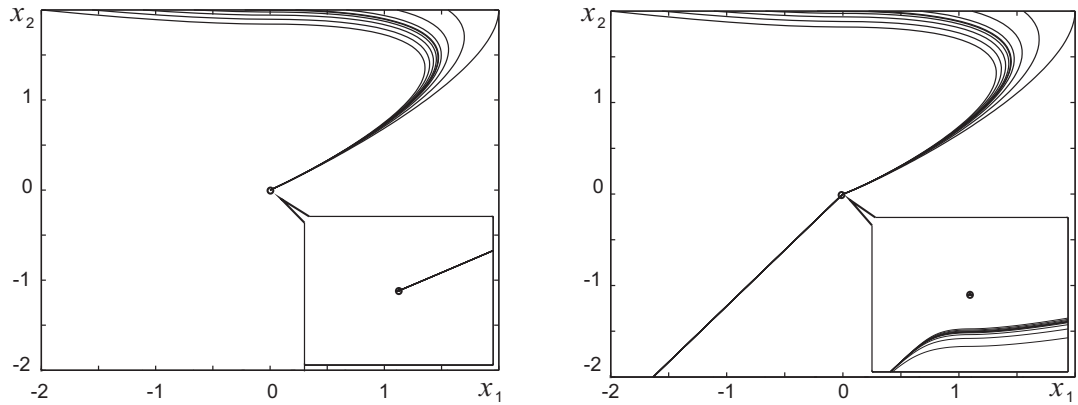}
\caption{Plots of the phase curves of system (\ref{eq:small_gain:example:1}) passing through the line $x_2(0)=2$. Parameters $\tau_1=1,c_1=1$ were fixed in all simulations, and parameter $c_2$ was varying. {Left panel:} phase curves of (\ref{eq:small_gain:example:1}) for $c_2=1/4$. Right panel: phase curves of the system for $c_2=11/40$. Insets in the bottom right corner of each plot show behavior of the phase curves within the rectangle $[-2\cdot 10^{-4},2\cdot 10^{-4}]\times[-2\cdot 10^{-4},2\cdot 10^{-4}]$, that is in the close proximity to the origin.}\label{fig:discussion}
\end{figure*}

Finally, we would also like to remark that  conditions
presented in e.g. Corollary \ref{cor:boundedness_unstable} can, in principle, be less
conservative than the ones established previously in the
framework of input-output/state analysis, cf.
\cite{SIAM_non_uniform_attractivity}. Indeed, when applied to
the same system, (\ref{eq:small_gain:example:1}), Corollary 4.1
from \cite{SIAM_non_uniform_attractivity} yields the following
upper bound for $c_2$:
\[
c_2<\frac{1}{16} \frac{\tau_1^2}{c_1},
\]
which is four times smaller than the one derived from Corollary
\ref{cor:boundedness_unstable}.

\end{example}

\section{Conclusion}\label{sec:conclusion}

In this manuscript we presented several results that are immediate consequences of our earlier published work \cite{Tyukin:SIAM:2013}. These results enabled to extend first method of Lyapunov for the analysis of asymptotic behavior of solutions in a vicinity of an equilibrium to systems in which the corresponding Jacobian has one zero eigenvalue. We showed that the fact that all other eigenvalues have negative real parts coupled with sign-definiteness condition of an associated quadratic form is sufficient to warrant that the equilibrium is a weak (Milnor) attractor. In particular, for  $\dot{z}=p(z)$ with $p:\Real^m\rightarrow\Real^m$ differentiable at least twice, let
\begin{itemize}
\item  $J$ be the Jacobian of the vector-filed $p$  at the origin,
\item  $\sigma_1\dots,\sigma_m$ be the eigenvalues of $J$ with $\sigma_1=0$ and $Re(\sigma_k)<0$, $k=2,\dots,m$,
\item  $T$ be a similarity transform such that the last row of $TJT^{-1}$ is zero
\item  $g(x)$ be the $m$-th component of $Tp(T^{-1} x)$, and
\item $G$ be the Hessian matrix of $g$.
\end{itemize}
 Then the origin is a (local) weak attractor if $G$ is sign-definite.

 Furthermore we provided analysis of convergence rates for a relevant subclass of systems with unstable attractors. We have shown that persistency of excitation plays an important role in establishing exponential convergence to the attractor. Finally, we demonstrated that conditions presented in the original work \cite{Tyukin:SIAM:2013} can be remarkably tight, at least for some example problems. Several questions, however, still remain. One of these is how (and if) the established rate of convergence may change in presence of unmodeled dynamics providing that the modulus is replaced with a dead-zone in (\ref{eq:system:observer}). Answering to these is the subject of ongoing work.


\bibliographystyle{IEEEtran}
\bibliography{Unstable_convergence_lemma}

\end{document}